\documentclass[preprint,11pt]{elsarticle}

\usepackage{amsmath,amssymb,amsthm}
\usepackage{graphicx}

\theoremstyle{plain}
\newtheorem{theorem}{Theorem}[section]
\newtheorem{proposition}[theorem]{Proposition}
\newtheorem{lemma}[theorem]{Lemma}
\newtheorem{corollary}[theorem]{Corollary}

\theoremstyle{definition}
\newtheorem{definition}[theorem]{Definition}

\newtheorem{remark}[theorem]{Remark}
\newtheorem{question}[theorem]{Question}

\newcommand{\K}{\mathcal{Q}}
\newcommand{\OO}{\mathcal{O}}

\newcommand{\N}{\mathbb{N}}
\newcommand{\PS}{P_S}
\newcommand{\UP}[2]{\uparrow_{#1}^{#2}}

\newcommand{\two}{\mathbf{2}}

\begin{document}

\begin{frontmatter}

\title{A note on power spaces and function spaces of strongly well-filtered spaces}

\author[1]{Zhenzhu Yuan}
\address[1]{School of Mathematical Sciences, Changsha Normal University, Changsha, Hunan, 410100, China}
\ead{yuanzhenzhuyzz@163.com}

\author[2]{Xiangrui Li}
\address[2]{School of Mathematics, Hunan University, Changsha, Hunan, 410082, China}
\ead{17731912783@163.com}

\author[3]{Bei Liu\corref{cor1}}
\address[3]{School of Mathematical Sciences, Hebei Normal University, Shijiazhuang  050024, China}
\cortext[cor1]{~Corresponding author.}
\ead{liubei@hebtu.edu.cn}

\author[2]{Qingguo Li}
\ead{liqingguoli@aliyun.com}
\begin{abstract}
	
 Strong well-filteredness is a refinement of well-filteredness arising
 naturally in non-Hausdorff topology and domain theory. In this paper,
 we construct a countable locally compact strong $R$-space $X$
 such that neither its Smyth power space nor its Scott power space
 is a strong $d$-space. This example gives negative answers to the
 questions concerning the preservation of strong well-filteredness
 and the strong $R$-space property under power space constructions.
 Furthermore, for the two-point discrete space $\two$, the function
 space $(\two,X)$ is not a strong $d$-space under any of the
 pointwise convergence, compact-open, Isbell, core-open
 or Scott topologies, where the Scott topology is induced by the pointwise order. 
\end{abstract}


\end{frontmatter}

\section{Introduction}\label{sec:intro}

Well-filtered spaces form an important class in non-Hausdorff topology and domain theory. It has been studied extensively in connection with posets and sober spaces \cite{Gierz,Goubault,LawsonXu2024,XiLawson2017,XuZhao2021}.

Strong $d$-spaces were introduced by Xu and Zhao \cite{XuZhao2020} as a
strengthening of the nation of $d$-spaces. To further investigate the relationship between well-filtered spaces and strong $d$-spaces, Xu \cite{Xu2026} introduced strongly well-filtered spaces.
Strong well-filteredness strengthens the defining condition of
well-filteredness by intersecting the filtered family with an additional
compact saturated set. Every strongly well-filtered space is both a
well-filtered space and a strong $d$-space. Every coherent well-filtered space is strongly well-filtered
\cite{Xu2026}. It is therefore natural to investigate whether this stronger property is preserved under the formation of power spaces and function spaces.

Power spaces provide one setting for this question. For a $T_0$-space $X$, let $\K(X)$ denote the family of all nonempty compact saturated subsets of $X$. Endowing $\K(X)$ with the upper Vietoris topology yields the Smyth power space $\PS(X)$. Equipping $\K(X)$ with the Scott topology induced by the reverse inclusion order gives the Scott power space $P_\sigma(X)=\Sigma\K(X)$. Both constructions play an important role in powerdomain theory and in the study of non-Hausdorff spaces \cite{Gierz,HeckmannKeimel}. The two topologies need not coincide in
general, but they do coincide when $X$ is locally compact and well-filtered
\cite{XuYang2021}. Xu, Wen and Xi \cite{XuWenXi2023} proved that the Scott
power space of a well-filtered space is again well-filtered. This result
motivates the corresponding preservation problem for strong
well-filteredness.

Xu posed the following question for Smyth power spaces.

\begin{question}[Question 7.3 of \cite{Xu2026}]\label{q1}
	For a strongly well-filtered space $X$, is $\PS(X)$ a strong $d$-space?
	Furthermore, is $\PS(X)$ strongly well-filtered?
\end{question}

We also consider the corresponding question for Scott power spaces.

\begin{question}[Question 7.7 of \cite{Xu2026}]\label{q2}
	For a strongly well-filtered space $X$, is its Scott power space
	$P_\sigma(X)$ strongly well-filtered?
\end{question}

A related class is that of strong $R$-spaces, introduced by Xu, Yang and
Chen \cite{XuYangChen2026} and studied further by Wen, Bao and Xu
\cite{WenBaoXu2026}. The strong $R$-space property requires that, whenever
the intersection of a nonempty family of compact saturated sets is
contained in an open set, the intersection of some nonempty finite
subfamily is already contained in that open set. Every strong $R$-space
is both an $R$-space and a strongly well-filtered space, and every coherent
well-filtered space is a strong $R$-space \cite{WenBaoXu2026}.
We also recall that an $R$-space is a strong $d$-space: applying property
$R$ to the family
$\{\mathord{\uparrow}d:d\in D\}\cup\{\mathord{\uparrow}x\}$ and using the
directedness of $D$ gives the defining implication for a strong $d$-space.

The power space questions for strong $R$-spaces are closely related to
those above. Wen, Bao and Xu \cite[Proposition 6.1]{WenBaoXu2026} proved
that if $\PS(X)$ is an $R$-space, then $X$ must be a strong $R$-space.
They also proved that $P_\sigma(X)$ is a coherent well-filtered space,
and hence, it is a strong $R$-space whenever $X$ is coherent and well-filtered
\cite[Proposition 6.6]{WenBaoXu2026}. They asked the following questions.

\begin{question}[Question 7.1 of \cite{WenBaoXu2026}]\label{q3}
	For a strong $R$-space $X$, is $\PS(X)$ an $R$-space?
	Furthermore, is $\PS(X)$ a strong $R$-space?
\end{question}

\begin{question}[Question 7.2 of \cite{WenBaoXu2026}]\label{q4}
	For a strong $R$-space $X$, is $P_\sigma(X)$ a strong $R$-space?
\end{question}

Function spaces give a second setting for preservation questions.
Topologies on sets of continuous maps, including the compact-open,
Isbell and Scott topologies, have been studied in connection with
sobriety, $d$-spaces and other order-theoretic properties
\cite{Isbell,Ershov2020,Liu2021,Ershov2022,Ershov2024}.
In particular, Xu, Bao and Zhang \cite{xubao2022} proved that the function
space of continuous maps into a well-filtered space is well-filtered
when equipped with the Isbell topology. This leads us to ask whether
strong well-filteredness or the strong $R$-space property is preserved
by the same construction.

In this paper, we construct a countable locally compact strong $R$-space
$X$ such that neither $\PS(X)$ nor $P_\sigma(X)$ is a strong $d$-space.
Thus both power spaces fail to be $R$-spaces, strongly well-filtered spaces
and strong $R$-spaces, giving negative answers to
Questions~\ref{q1}--\ref{q4}. The construction uses the two Scott spaces
$A$ and $B$ from the product counterexample discussed in
\cite[Example 5.5]{Xu2026} and \cite[Example 5.5]{WenBaoXu2026}.
We prove that arbitrary topological sums of strong $R$-spaces are strong
$R$-spaces and take $X=A\sqcup B$. A closed subspace of $\PS(X)$ is
homeomorphic to $\PS(A)\times\PS(B)$, where we construct an explicit
failure of the strong $d$-space condition. Local compactness and
well-filteredness then allow us to transfer the result to $P_\sigma(X)$.

For function spaces, we use the two-point discrete space $\two$.
The Isbell function space $[\two,X]_{\mathcal I}$ is homeomorphic to
$X\times X$, which is not a strong $d$-space. We further show that the
same conclusion holds for the pointwise convergence, compact-open,
core-open and Scott topologies, where the Scott topology is defined
using the pointwise order. Thus the failure of preservation already
occurs for a finite discrete domain and a countable locally compact
strong $R$-space as a codomain.

\section{Preliminaries}\label{sec:prelim}

We recall the notions needed in the sequel; the reader is referred to
\cite{Gierz, Goubault, Xu2026} for more details.

For a poset $P$ and $A\subseteq P$, we write
$\mathord{\uparrow}A=\{x\in P:\exists a\in A,\ a\leq x\}$ and
$\mathord{\downarrow}A=\{x\in P:\exists a\in A,\ x\leq a\}$; for $x\in P$, we
write $\mathord{\uparrow}x$ and $\mathord{\downarrow}x$ for the corresponding
principal sets. A subset $A$ is an \emph{upper set} (resp.\ \emph{lower set})
if $A=\mathord{\uparrow}A$ (resp.\ $A=\mathord{\downarrow}A$). A nonempty
subset $D\subseteq P$ is \emph{directed} if any two elements of $D$ have an
upper bound in $D$. A poset in which every directed subset has a supremum is a
\emph{dcpo}. A poset is \emph{Noetherian} if every directed subset has a
largest element.

A subset $U$ of a dcpo $P$ is \emph{Scott open} if $U=\mathord{\uparrow}U$ and,
for every directed subset $D\subseteq P$ with $\bigvee D$ existing, $\bigvee
D\in U$ implies $D\cap U\neq\emptyset$. The Scott open sets form the
\emph{Scott topology} $\sigma(P)$, and $\Sigma P=(P,\sigma(P))$ denotes the
Scott space of $P$. The \emph{Alexandroff topology} $\alpha(P)$ consists of
all upper sets of $P$. If every directed subset of $P$ has a largest element,
then every upper set is Scott open, whence $\alpha(P)=\sigma(P)$.

Throughout this paper, $X$ denotes a $T_0$-space. The \emph{specialization order}
$\leq_X$ on $X$ is defined by $x\leq_X y$ iff $x\in\overline{\{y\}}$. A subset
$A$ of $X$ is \emph{saturated} if it is an upper set with respect to
$\leq_X$, i.e.,\ if it equals the intersection of all open sets containing it.
For $x\in X$, the saturation $\mathord{\uparrow}x$ of $\{x\}$ is an upper set.
We denote by $\K(X)$ the family of all \emph{nonempty} compact saturated
subsets of $X$.


The set of all natural numbers is denoted by $\mathbb{N}$, the order on $\mathbb{N}$ is the usual order of natural numbers. Let $\mathbb{N}^+=\mathbb{N}\setminus\{0\}$.

A family $\{K_d:d\in D\}\subseteq\K(X)$ indexed by a directed set $D$ is
called \emph{filtered} if for all $d_1,d_2\in D$ there exists $d_3\in D$ with
$K_{d_3}\subseteq K_{d_1}\cap K_{d_2}$. A $T_0$-space $X$ is \emph{well-filtered} if for every
filtered family $\{K_d:d\in D\}\subseteq\K(X)$ and every open set $U$ of $X$,
\[
\bigcap_{d\in D}K_d\subseteq U
\]
implies $K_d\subseteq U$ for some $d\in D$. A $T_0$-space $X$ is \emph{sober}
if every irreducible closed set is the closure of a unique point. Every sober
space is well-filtered \cite{Gierz}.

\begin{definition}\cite[Definition 3.18]{XuZhao2020}\label{df1}
A $T_0$-space $X$ is a \emph{strong
$d$-space} if for every directed subset $D\subseteq X$ (with respect to
$\leq_X$), every $x\in X$ and every open set $U$ of $X$,
$
\bigcap_{d\in D}\mathord{\uparrow}d\ \cap\ \mathord{\uparrow}x\subseteq U
$
implies $\mathord{\uparrow}d\cap\mathord{\uparrow}x\subseteq U$ for some
$d\in D$.
\end{definition}

\begin{definition}\cite[Definition 4.1]{Xu2026}
 A $T_0$-space $X$ is
\emph{strongly well-filtered} if for every filtered family
$\{K_d:d\in D\}\subseteq\K(X)$, every $K\in\K(X)$ and every open set $U$ of
$X$,
$
\bigcap_{d\in D}K_d \cap K\subseteq U
$
implies $K_d\cap K\subseteq U$ for some $d\in D$.
\end{definition}

For an open set $U$ of $X$, we write
$\square U=\{K\in\K(X):K\subseteq U\}.$ The topology on $\K(X)$ generated by $\{\square U:U\in\OO(X)\}$ is the
\emph{upper Vietoris topology}, and the resulting space is the \emph{Smyth
power space} $\PS(X)$ of $X$ (see \cite{HeckmannKeimel, Xu2026}). Its
specialization order is the reverse inclusion order:
\[
K_1\leq_{\PS(X)}K_2\iff K_2\subseteq K_1.
\]
Consequently, for $K\in\K(X)$ the principal upper set of $K$ in $\PS(X)$ is
\[
\UP{\PS(X)}{}K=\{L\in\K(X):L\subseteq K\},
\]
i.e.,\ the family of all nonempty compact saturated subsets of $K$ that belong
to $\K(X)$. When $\K(X)$ is ordered by reverse inclusion, the space
$P_\sigma(X)=\Sigma\K(X)$ is called the \emph{Scott power space} of $X$.

Write $\mathrm{Fin}\,X=\{\mathord{\uparrow}F:\emptyset\neq F\subseteq X,\ F
\text{ finite}\}$.

\begin{definition}\label{def:Rspace}
\begin{enumerate}
\item[(1)] (\cite[Definition 10.2.11]{Xu2016}) A
$T_0$-space $X$ is said to have \emph{property $R$} if for any nonempty family
$\{\mathord{\uparrow}F_i:i\in I\}\subseteq\mathrm{Fin}\,X$ and any open set
$U$ of $X$,
\[
\bigcap_{i\in I}\mathord{\uparrow}F_i\subseteq U
\]
implies $\bigcap_{i\in J}\mathord{\uparrow}F_i\subseteq U$ for some nonempty
finite $J\subseteq I$. A $T_0$-space satisfying property $R$ is called an
\emph{$R$-space}.
\item[(2)] (\cite[Definition 4.2]{XuYangChen2026}) A
$T_0$-space $X$ is called a \emph{strong $R$-space} if for any nonempty family
$\{K_i:i\in I\}\subseteq\K(X)$ and any open set $U$ of $X$,
\[
\bigcap_{i\in I}K_i\subseteq U
\]
implies $\bigcap_{i\in J}K_i\subseteq U$ for some nonempty finite
$J\subseteq I$.
\end{enumerate}
\end{definition}

\begin{remark}
	It was shown in \cite{WenBaoXu2026} that every strong $R$-space is strongly well-filtered, and every strongly well-filtered space is a strong $d$-space.
\end{remark}

We shall use the following facts from \cite{Xu2026}.

\begin{theorem}[\cite{Gierz} Theorem II-4.13, \cite{Xu2026} Lemma 1.9]\label{tmscott}
Let \(P\) be a poset. Then the following statements are equivalent:
\begin{enumerate}
    \item[(1)] \(\Sigma P\) is core-compact.
    \item[(2)] For every poset \(S\), one has \(\Sigma(P \times S) = \Sigma P \times \Sigma S\), that is, the Scott topology of \(P \times S\) is equal to the product of the individual Scott topologies.
    \item[(3)] For every dcpo or complete lattice \(S\), one has \(\Sigma(P \times S) = \Sigma P \times \Sigma S\).
    \item[(4)] \(\Sigma(P \times \sigma(P)) = \Sigma P \times \Sigma \sigma(P)\).
\end{enumerate}
\end{theorem}

The following results are well-known (cf. \cite{XuYang2021},  \cite{Xu2026}, \cite{WenBaoXu2026}).

\begin{proposition}\label{prop:retractive}
A retract of a strong $d$-space (strongly well-filtered space, strong $R$-space, respectively) is a strong $d$-space (strongly well-filtered space, strong $R$-space, respectively).
\end{proposition}

\begin{lemma}\label{lem:coincide}
If $X$ is a locally compact well-filtered space, then the upper Vietoris
topology and the Scott topology on $\K(X)$ coincide; hence,
$\PS(X)=P_\sigma(X)$.
\end{lemma}

\begin{proposition}\label{lem:sd-closed}
Let \(X\) be a strong $d$-space (strongly well-filtered space, strong $R$-space, respectively) and \(A\) a nonempty closed set of \(X\). Then \(A\) as a subspace of \(X\) is a strong $d$-space (strongly well-filtered space, strong $R$-space, respectively).
\end{proposition}

\begin{theorem}[\cite{En}, Theorem 3.3.12]\label{th1}
The sum $\bigsqcup_{s\in S} X_s$ is locally compact if and only if all spaces $X_s$ are locally compact.
\end{theorem}

\begin{definition}[\cite{Gierz}, \cite{Goubault}, \cite{Ershov2024}]
For two spaces \(X\) and \(Y\), let \((X,Y)\) be the set of all continuous functions from $X$ to $Y$.

(1) Let \(H\) be a Scott open subset of \(\mathcal{O}(X)\), \(V\) an open subset of \(Y\), and set
$$N(H \leftarrow V)=\{f \in TOP(X,Y): f^{-1}(V) \in H\}.$$
As \(H\) ranges over \(\sigma(\mathcal{O}(X))\) and \(V\) ranges over \(\mathcal{O}(Y)\), the sets \(N(H \leftarrow V)\) form a subbasis for a topology on \((X,Y)\) called the \textit{Isbell topology}. Let $[X,Y]_{\mathcal{I}}$ denote the set \((X,Y)\) endowed with the Isbell topology.

(2) The collection of all sets of the form
\[
V_{x,U} = \{f \in (X,Y) \mid f(x) \in U\}, \text{ where } x \in X \text{ and } \emptyset \neq U \in \mathcal{O}(Y)),
\]
forms a subbasis of the \textit{pointwise convergence topology}. We denote by $[X,Y]_{\mathcal{P}}$ the set $(X,Y)$ endowed with the pointwise convergence topology.

(3) The \textit{compact-open topology} on $(X,Y)$ is defined by the subbasis of open sets
$$
V_{K,U} = \{f \in (X,Y) \mid f(K) \subseteq U\}, \text{ where } K \subseteq X \text{ is compact and } \emptyset \neq U \in \mathcal{O}(Y),
$$
let $[X,Y]_{\mathcal{K}}$ denote the set \((X,Y)\) endowed with the compact-open topology.

(4) The \textit{core-open topology} on \((X,Y)\) is defined by the subbasis of open sets
$$
V_{U,W} = \{f \in (X,Y) \mid U \ll f^{-1}(W)\}, \text{ where } \emptyset \neq U \in \mathcal{O}(X) \text{ and } \emptyset \neq W \in \mathcal{O}(Y),
$$
let $[X,Y]_{\mathcal{C}}$ denote the set \((X,Y)\) endowed with the core-open topology.

(5) The \textit{natural topology} on \((X,Y)\) consists of those sets $U \subseteq (X,Y)$ which satisfy the following condition for all nets \(\mathfrak{f}\) in \((X,Y)\) and all \(f \in (X,Y)\):
\[
\text{if } \mathfrak{f} \text{ converges continuously to } f \in U \text{ then } \mathfrak{f} \text{ is eventually in } U,
\]
let $[X,Y]_{\mathcal{T}_*}$ denote the set \((X,Y)\) endowed with the natural topology.

(6) Consider the topology on \((X, Y)\) defined by the subbasis of open sets
\[
V_{\mathcal{H},U}=\{f\in (X, Y)\mid f^{-1}(U)\in\mathcal{H}\},
\]
where \(\mathcal{H}\) is an Alexandroff open subset of \(\mathcal{O}(X)\) and \(U\in\mathcal{O}(Y)\). We denote it by \(\mathcal{T}_{A}(\subseteq)^{\#}\).
\end{definition}

\begin{definition} [\cite{Ershov2020}]
A space \(X\) is an \(\alpha^*\)-space, if for each set \(U \in \mathcal{O}(X)\) and each element \(x \in U\), there are elements \(x_0, \ldots, x_n \in U\) such that \(x \in \operatorname{int}(\uparrow\!x_0 \cup \ldots \cup \uparrow\!x_n)\).
\end{definition}

For \(T_0\)-spaces \(X\) and \(Y\), for an element \(x \in X\), and for a topology \(\mathcal{T}\) on \((X,Y)\), we consider the following mappings:
$$
\xi: Y \to (X, Y)_{\mathcal{T}}, \quad \xi: y \mapsto \xi_y, \quad \text{where } \xi_y(x) = y \text{ for all } x \in X;$$
$$\zeta_{x,\mathcal{T}}: (X, Y)_{\mathcal{T}} \to Y, \quad \zeta_{x,\mathcal{T}}: f \mapsto f(x).
$$

\begin{lemma}[\cite{Ershov2022} Lemma 3, \cite{Ershov2024} Lemma 4]\label{le4}
Let \(X, Y\) be \(T_0\)-spaces.

(1) If \(\mathcal{T}\) is a topology on \((X, Y)\) such that \(\mathcal{P} \subseteq \mathcal{T} \subseteq \mathcal{T}_{A}(\subseteq)^{\#}\), then \(\xi\) is a homeomorphic embedding of \(Y\) into \((X, Y)_{\mathcal{T}}\).

(2) If \(\mathcal{T}\) is a topology on \((X, Y)\) such that \(\mathcal{P} \subseteq \mathcal{T}\), then \(\zeta_{x,\mathcal{T}}\) is continuous.
\end{lemma}

\begin{proposition}[\cite{Ershov2024}, Proposition 1(4) and 3(3)]\label{proer}
	Let \(X, Y\) be \(T_0\)-spaces. 
	
	(1) \(\mathcal{P} \subseteq \mathcal{K} \subseteq \mathcal{I} \subseteq \mathcal{T}_* \text{ and } \mathcal{K} \subseteq \mathcal{C}\).
	
	(2) If \(X\) is an \(\alpha^*\)-space, then \(\mathcal{P} = \mathcal{K} = \mathcal{C} = \mathcal{I} = \mathcal{T}_*\) and this topology is exponential on $(X, Y)$.
\end{proposition}

\begin{remark}\label{recon}
For $\mathcal{T}=\mathcal{P}, \mathcal{K}, \mathcal{I}$ and $\mathcal{C}$, the Proposition~\ref{proer} and Lemma~\ref{le4} imply that \(\zeta_{x,\mathcal{T}}\) is continuous for each $x\in X$. It is easy to prove that $\mathcal{I}, \mathcal{K}, \mathcal{C}\subseteq \mathcal{T}_{A}(\subseteq)^{\#}$, so \(\xi\) is continuous.
\end{remark}

\section{Smyth and Scott power space of a strong R-space are not always strong d-spaces }

We construct a counterexample by Example 5.5 in References \cite{Xu2026} and \cite{WenBaoXu2026}. Throughout the paper, let $\mathbb{N}_\infty^+ = \mathbb{N}^+ \cup \{\infty\}$ be the chain by adding a top element $\infty$ to $\mathbb{N}^+$, and $Q=\{\omega_n: n\in \mathbb{N}^+\}\cup\{a,b\}$. Define an order on $Q$ as follows (see Fig.~\ref{fig:chain_noetherian_dcpo}):

(1) $a<\omega_n$ and $b<\omega_n$ for all $n\in \mathbb{N}^+$,

(2) for any $n,m\in \mathbb{N}^+$ with $n\neq m$, $\omega_n$ and $\omega_m$ are incomparable, and

(3) $a$ and $b$ are incomparable.

\begin{figure}[htbp]
\centering
\includegraphics[width=0.65\textwidth]{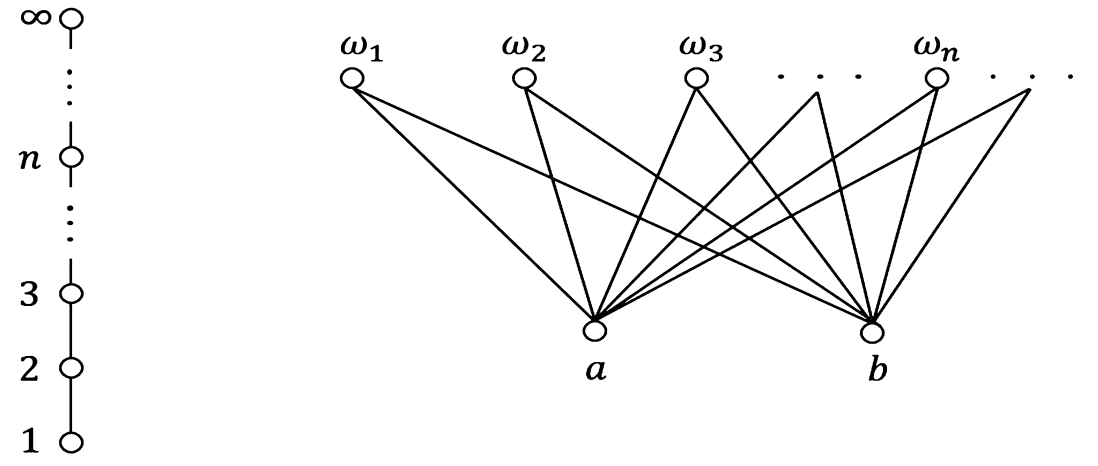}
\caption{The countable complete chain $\mathbb{N}_\infty^+$ and countable Noetherian dcpo $Q$}
\label{fig:chain_noetherian_dcpo}
\end{figure}

For convenience, we use $A$ and $B$ to stand for the Scott spaces $\Sigma\mathbb{N}_\infty^+$ and $\Sigma Q$, respectively.

\begin{remark}\label{re2}
We know that  $A$ and $B$ are strong $R$-spaces; thus they are both strongly well-filtered and strong $d$-spaces from Example 5.5 in~\cite{Xu2026} and~\cite{WenBaoXu2026}. Furthermore, the product space $A\times B$ is not a strong $d$-space.
\end{remark}

\begin{remark}\label{re1}
For each $n\in\N^{+}$, set
$$T_n=\mathord{\uparrow} n=\{k\in\N^{+}:k\geq n\}\cup\{\infty\},~T_\infty=\{\infty\}$$
$$E_a=\mathord{\uparrow} a=\{a\}\cup\Omega,
~E_b=\mathord{\uparrow} b=\{b\}\cup\Omega,
~\Omega=\{\omega_n:n\in\N^{+}\}.
$$
Then $\K(A)=\{T_n:n\in\N^{+}\}\cup\{T_\infty\}$. Moreover, each $T_n$ is a Scott open subset of $A$, and $\bigcap_{n\geq 1}T_n=T_\infty=\{\infty\}$.
Since every directed subset of $Q$ has a largest element, we have $\sigma(Q)=\alpha(Q)$.
Clearly, the compact saturated subsets of $B$ are exactly the
nonempty upper sets generated by finite sets, i.e., $\K(B)=\{\mathord{\uparrow} F:\emptyset\neq F\subseteq Q,\ F \text{ finite}\}$; $E_a\cap E_b=\Omega$, and $\Omega$ is not compact in $B$.
\end{remark}

\begin{remark}
Let $X=A\sqcup B$ be the sum of the spaces $A$ and $B$. Then the topology on \(X\) is actually the Scott topology. $A$ and
$B$ are clopen subsets of $X$, and a nonempty subset $K\subseteq X$ is compact
saturated in $X$ if and only if $K\cap A$ and $K\cap B$ are compact saturated
subsets of $A$ and $B$, respectively. In
particular, every element of $\K(X)$ is a disjoint union $H\cup L$ with
$H\in\K(A)\cup\{\emptyset \}$ and $L\in\K(B)\cup\{\emptyset \}$, where $H$ and $L$ are not both empty.
\end{remark}

\begin{theorem}\label{thm:sum-infinite}
Let $\Gamma$ be an arbitrary nonempty index set, and let
$\{X_\gamma:\gamma\in\Gamma\}$ be a family of strong $R$-spaces. Then the
 sum space $X^{'}=\bigsqcup_{\gamma\in\Gamma}X_\gamma$ is a strong $R$-space.
\end{theorem}

\begin{proof}
Note that every $X_\gamma$ is clopen of $X^{'}$ for $\gamma\in\Gamma$, and $\{X_\gamma:\gamma\in\Gamma\}$
is an open cover of $X^{'}$.  For any nonempty family
$\{K_i:i\in I\}\subseteq\K(X^{'})$ and an open set $U$ of $X^{'}$, if
$\bigcap_{i\in I}K_i\subseteq U$, then there exist
$K_i^\gamma=K_i\cap X_\gamma\in\K(X_\gamma)\cup\{\emptyset\}$,
$U_\gamma=U\cap X_\gamma\in\mathcal{O}(X_\gamma)$ such that $\bigcap_{i\in I}K_i^\gamma
=\Big(\bigcap_{i\in I}K_i\Big)\cap X_\gamma\subseteq U\cap X_\gamma
=U_\gamma$. Since $K_i=\bigsqcup_{\gamma\in\Gamma}K_i^\gamma$ is a
disjoint union for each $i\in I$, we have
$\bigcap_{i\in I}K_i
=\bigsqcup_{\gamma\in\Gamma}\Big(\bigcap_{i\in I}K_i^\gamma\Big)$. To prove that $X^{'}$ is a strong $R$-space, it suffices to find a nonempty finite set $F\subseteq I$ such that $\bigcap_{i\in F}K_i\subseteq U$.

Let $\operatorname{supp}(i)=\{\gamma\in\Gamma:K_i^\gamma\neq\emptyset\}$.
As $K_i$ is compact and $\{X_\gamma:\gamma\in\Gamma\}$ is an open cover of $K_i$, the set $\operatorname{supp}(i)$ is finite for every $i\in I$. Define
$$
\Gamma_{\ast}:=\{\gamma\in\Gamma:K_i^\gamma\neq\emptyset\text{ for all }
i\in I\}=\bigcap_{i\in I}\operatorname{supp}(i).
$$
As $I\neq\emptyset$, fixing $i_0\in I$ gives
$\Gamma_{\ast}\subseteq\operatorname{supp}(i_0)$, so
$\Gamma_{\ast}$ is finite.

If $\Gamma_{\ast}=\emptyset$,
for any $\gamma\in\operatorname{supp}(i_0)$, then there is a $i_\gamma\in I$ such that $K_{i_\gamma}^\gamma=\emptyset$. It implies $K_{i_0}\cap\bigcap_{\gamma\in \operatorname{supp}(i_0)}K_{i_\gamma}=\emptyset\subseteq U$.

If $\Gamma_{\ast}\neq\emptyset$, for each $\gamma\in\operatorname{supp}(i_0)$, then $\bigcap_{i\in I}K_i^\gamma\subseteq U_\gamma.$ Hence, there is a finite subset $F_\gamma\subseteq I$ such that $\bigcap_{i\in F_\gamma}K_i^\gamma\subseteq U_\gamma$ as $X_\gamma$ is a strong R-space. Let $F=\bigcup_{\gamma\in\operatorname{supp}(i_0)}F_\gamma$. It follows that $$K_{i_0}\cap\bigcap_{i\in F}K_i=\bigsqcup_{\gamma\in\Gamma}\Big(\bigcap_{i\in F\cup\{i_0\}}K_i^\gamma\Big)\subseteq U.$$

Therefore, $X^{'}=\bigsqcup_{\gamma\in\Gamma}X_\gamma$ is a strong $R$-space.
\end{proof}

\begin{remark}
Theorem \ref{thm:sum-infinite} also holds for strongly well-filtered spaces or strong $d$-spaces, and the proof proceeds similarly.

\end{remark}

Since both $A$ and $B$ are strong $R$-spaces,
Theorem \ref{thm:sum-infinite}  yields the following.

\begin{corollary}\label{cor:X-swf}
The space $X=A\sqcup B$ is a strong $R$-space; in particular, $X$ is
strongly well-filtered and a strong $d$-space.
\end{corollary}

We shall also use the following elementary observations about $X$.

\begin{proposition}\label{lem:X-prop}
The space $X$ is locally compact and not coherent.
\end{proposition}

\begin{proof}
 Obviously, $A$ is the Scott space of the algebraic complete lattice $\mathbb{N}_\infty^+$; thus, it is locally compact. For $a$, $b$ or every $\omega_n$ for $n\in \mathbb{N}^+$, there exist corresponding compact saturated open sets
${\uparrow}a$, ${\uparrow}b$ or $\{\omega_n\}$ such that $B$ is locally compact. By Theorem~\ref{th1}, $X$ is locally compact. Clearly, $E_a$ and $E_b$ are compact saturated
in $B$, but $E_a\cap E_b=\Omega$ is not compact by Remark~\ref{re1}; hence, the space $X$ is not coherent.
\end{proof}

Next, we give two proofs to show that $P_S(X)$ is not a strong $d$-space. The first is indirect and can also be used to investigate other properties of the space, while the second gives a direct proof.

\begin{proposition}\label{eq:C}
Let $\mathcal{C}=\{K\in\K(X):K\cap A\neq\emptyset\text{ and }K\cap B\neq\emptyset\}$. Then $\mathcal{C}$ is a closed subspace of $\PS(X)$.
\end{proposition}

\begin{proof} Since $K$ is formed by the disjoint union of compact saturated sets from $A$ and $B$, we have $K\notin \mathcal{C}$ if and only if $K\cap A=\emptyset$ or $K\cap B=\emptyset$, which is equivalent to $K\subseteq B$ or $K\subseteq A$, i.e., $K\in \square B$ or $K\in \square A$. Thus $\PS(X)\setminus \mathcal{C}=\square A\cup\square B$. Hence, $\mathcal{C}$ is a closed subspace of $\PS(X)$ as $A$ and $B$ are both open subsets of $X$.
\end {proof}

\begin{lemma}\label{prop:C-homeo}
The map $\Phi:\mathcal{C}\longrightarrow \PS(A)\times\PS(B)$ with
$\Phi(K)=(K\cap A,\;K\cap B)$ is a homeomorphism.
\end{lemma}

\begin{proof}
Clearly, for each $K\in \mathcal{C}$, both $K\cap A$ and $K\cap B$ are nonempty compact saturated subsets of $A$ and $B$, respectively. So $\Phi$ is well-defined and injective. Conversely,
for $(H,L)\in\PS(A)\times\PS(B)$, the disjoint union $H\cup L$ is compact
saturated in $X$ and meets both $A$ and $B$, and $\Phi(H\cup L)=(H,L)$; hence,
$\Phi$ is a bijection whose inverse is $\Psi(H,L)=H\cup L$. We claim that $\Phi$ is continuous and open. For nonempty open subsets $U_A\in\OO(A)$ and $U_B\in\OO(B)$, the set $U_A\cup U_B$ is open in $X$, and
\begin{align*}
\Phi\big(\mathcal{C}\cap\square_X(U_A\cup U_B)\big) &=\Phi(\{K\in \mathcal{C}:K\subseteq U_A\cup U_B\})\\
&                                                    =\{(K\cap A, K\cap B): K\subseteq U_A\cup U_B,K\in\mathcal{C}\}\\
&                                                    =\square_A U_A\times\square_B U_B.
\end{align*}

Otherwise, assume $U_A$ or $U_B$ is empty, then $\Phi\big(\mathcal{C}\cap\square_X(U_A\cup U_B)\big)=\emptyset$.

Indeed, $K\in \mathcal{C}$ satisfies $K\subseteq U_A\cup U_B$ if and only if
$K\cap A\subseteq U_A$ and $K\cap B\subseteq U_B$. The sets
$\square_A U_A\times\square_B U_B$ form a base of
$\PS(A)\times\PS(B)$, while the sets $\mathcal{C}\cap\square_X U$ ($U\in\OO(X)$) form a
base of the subspace topology of $\mathcal{C}$. Thus, $\Phi$ is a homeomorphism.
\end{proof}

For convenience, Let $Y=\PS(A)\times\PS(B)$.
The specialization order of $Y$ is the componentwise reverse inclusion order, i.e.,
$(H,L)\leq_Y(H',L')$ if and only if $H'\subseteq H$ and $L'\subseteq L$. For
$z\in Y$, the principal filter is denoted by
$\UP{Y}{}z=\{z'\in Y:z\leq_Y z'\}$.

\begin{proposition}\label{prop:Y-not-sd}
$Y=\PS(A)\times\PS(B)$ is not a strong $d$-space.
\end{proposition}

\begin{proof}
Define $d_n=(T_n,E_a)\in Y$  for every $n\in \N^{+}$,
and $x=(T_1,E_b)\in Y$, where $T_n$, $E_a$ and $E_b$ are those as in Remark~\ref{re1}. Since $T_{n+1}\subseteq T_n$ for every $n$, we have
$d_n\leq_Y d_{n+1}$ for all $n\in \N^{+}$. It follows that
$D=\{d_n:n\in\N^{+}\}$
is a directed subset of $Y$. We verify that the directed set $D$ and the point $x$
violate the defining condition of a strong $d$-space.

Note that for each $(H,L)\in Y$,
$(H,L)\in \bigcap_{n\in\N^{+}}{\UP{Y}{}d_n}\cap{\UP{Y}{}x}$
if and only if $d_n\leq_Y(H,L)$ and $x\leq_Y(H,L)$ for all $n\in\N^{+}$, that is,
$H\subseteq\bigcap_{n\in\N^{+}}T_n\cap T_1=\{\infty\}$ and
$L\subseteq E_a\cap E_b=\Omega$. As $H$ is nonempty, $H=T_\infty=\{\infty\}$; by Remark ~\ref{re1}, $L$ ranges exactly over the nonempty finite subsets of $\Omega$. Therefore,
$$\bigcap_{n\in\N^{+}}{\UP{Y}{}d_n}\cap{\UP{Y}{}x}
=\big\{(\{\infty\},F):F\subseteq\Omega,\ 0<|F|<\infty\big\}.$$
Let $M(F)=\max\{j\in\N^{+}:\omega_j\in F\}$ for any
nonempty finite set $F\subseteq\Omega$,
define the subset
$$\mathcal U=\bigcup_{\substack{F\subseteq\Omega\\0<|F|<\infty}}
\Big(\square_A T_{M(F)}\times\square_B F\Big)$$
of $Y$. Since $T_{M(F)}$, $F$ are Scott open sets of $A$ and $B$, respectively, we obtain that $\mathcal U$ is an open set of $Y$. For every nonempty finite $F\subseteq\Omega$,
$(\{\infty\},F)\in\square_A T_{M(F)}\times\square_B F\subseteq\mathcal U$ by $\{\infty\}\subseteq T_{M(F)}$ and $F\subseteq F$,
it implies that
$\bigcap_{n\in\N^{+}}{\UP{Y}{}d_n}\cap{\UP{Y}{}x}\subseteq\mathcal U$.

Fix $n\in\N^{+}$. Then $(T_{n},\{\omega_{n+1}\})\in {\UP{Y}{}d_n}\cap{\UP{Y}{}x}$. But $(T_{n},\{\omega_{n+1}\})\not\in \mathcal{U}$. Otherwise, $T_{n}\in \square_A T_{M(F)}$ implies that $M(F)\le n$, it contradicts $M(F)=M(\{\omega_{n+1}\})=n+1>n$. Hence, $Y$ is not a strong $d$-space by Definition~\ref{df1}.
\end{proof}

\begin{corollary}\label{cor:Y-not-swf}
$Y=\PS(A)\times\PS(B)$  is neither strongly well-filtered nor a strong $R$-space.
\end{corollary}

Now, we can directly show that \(P_S(X)\) is not a strong \(d\)-space.

\begin{proposition}
	\(Z=P_S(X)\) is not a strong \(d\)-space.
\end{proposition}

\begin{proof}
	Let $M(F)=\max\{j\in\N^{+}:\omega_j\in F\}$ for any
	nonempty finite set $F\subseteq\Omega$, define subsets 
	$$\mathcal U=\bigcup_{\substack{F\subseteq\Omega\\0<|F|<\infty}}
	\square (T_{M(F)}\cup F), $$ for each $n\in\mathbb{N^+}$,	$\uparrow_Z d_n=\uparrow_Z (T_n\cup E_a),~\uparrow_Z d=\uparrow_Z (T_1\cup E_b).$ We take the openness of $\mathcal{U}$, by $\square (T_{M(F)}\cup F)\in\mathcal{O}(P_S(X))$ for each nonempty finite subset $F\subseteq\Omega$. Then $\{\uparrow_Z d_n:n\in\mathbb{N}^+\}$ is directed.  
	
	Claim 1. $\bigcap_{n\in\mathbb{N}^+}\uparrow_Z d_n\cap\uparrow_Z d=\{\{\infty\}\cup F:F\subseteq\Omega,F~is~finite\}\cup\{F\subseteq \Omega:0<|F|<\infty\}.$
	
	For each $K\in Z$, $K\in \bigcap_{n\in\mathbb{N}^+}\uparrow_Z d_n\cap\uparrow_Z d$ if and only if $K\cap A\subseteq \bigcap_{n\in\mathbb{N^+}}T_n=\{\infty\}$, $K\cap B\subseteq E_a\cap E_b=\Omega$. Hence, the equality follows.
	
	Claim 2. $\bigcap_{n\in\mathbb{N}^+}\uparrow_Z d_n\cap\uparrow_Z d\subseteq\mathcal{U}.$
	
	Note that $\infty\in T_n$ for each $n\in\mathbb{N}^+$. Hence, for each nonempty finite subset $F\subseteq \Omega$, $$F\subseteq (\{\infty\}\cup F)\subseteq  \square (T_{M(F)}\cup F)\subseteq \mathcal{U}.$$ Specially, $\{\infty\}\in\square (T_{M(F)}\cup F)$, for each nonempty finite subset $F\subseteq \Omega$.
	
	Claim 3. For each $n\in\mathbb{N}^+$, $\uparrow_Z d_n\cap\uparrow_Z d\not\subseteq\mathcal{U}.$
	
	Given $n\in\mathbb{N}^+$, $T_{n}\cup\{\omega_{n+1}\}\in {\UP{Z}{}d_n}\cap{\UP{Z}{}d}$. If $T_{n}\cup\{\omega_{n+1}\}\in \mathcal{U}$, then $T_{n}\cup\{\omega_{n+1}\}\in \square (T_{M(F)}\cup F)$ for some nonempty finite subset $F\subseteq \Omega$. It follows that $T_{n}\subseteq T_{M(F)}$ and $\{\omega_{n+1}\}\subseteq F$. It implies that $M(F)\le n$. But it contradicts $M(F)\geq M(\{\omega_{n+1}\})=n+1>n$. Hence, $Z=P_S(X)$ is not a strong $d$-space by Definition~\ref{df1}.
\end{proof}

We can now assemble the main result, and obtain negative answers to Questions \ref{q1}, \ref{q2}, \ref{q3} and \ref{q4} in References~\cite{Xu2026} and ~\cite{WenBaoXu2026}.

\begin{theorem}\label{thm:main}
Both $\PS(X)$  and $P_\sigma(X)$ fail to be strong \(d\)-spaces, in particular, not strongly well-filtered or a strong $R$-space.
\end{theorem}

\begin{proof}
Combining Proposition~\ref{lem:sd-closed}, Proposition~\ref{eq:C}, Lemma~\ref{prop:C-homeo} and Proposition~\ref{prop:Y-not-sd}, we conclude that $\PS(X)$ is not a strong \(d\)-space. On the other hand, by Lemma~\ref{lem:coincide} and Lemma~\ref{lem:X-prop},   $\PS(X)=P_\sigma(X)$; hence, $P_\sigma(X)$ is not a strong \(d\)-space either.
\end{proof}

\section{Function spaces equipped with Isbell topology}

In this section, we mainly discuss the function spaces of strong $R$-spaces, strongly well-filtered and strong $d$-spaces. To this end, we first introduce some notations and properties. $X=A\sqcup B$ still denotes the space mentioned above.

\begin{proposition}\label{prop:Xp-not-sd}
The space $X\times X$ is not a strong $d$-space.
\end{proposition}

\begin{proof}
Suppose that $X\times X$ is a strong $d$-space, then every nonempty closed subspace of $X\times X$ is also a strong $d$-space by Lemma~\ref{lem:sd-closed}, which  contradicts Remark~\ref{re2}. Consequently, $X\times X$ is not a strong $d$-space.
\end{proof}

\begin{proposition}\label{prop:Xsq-not-sd}
Let $T, Z$ be $T_{0}$-spaces and $[T,Z]_\mathcal{T}$ be strong $d$-space \rm{(strongly well-filtered, strong $R$-space, respectively)}, where $\mathcal{T}=\mathcal{P}, \mathcal{K}, \mathcal{I}$ or $\mathcal{C}$. Then $Z$ is a strong $d$-space (strongly well-filtered space, strong $R$-space, respectively).
\end{proposition}

\begin{proof}
 By Remark~\ref{recon}, we obtain
$\zeta_{x,\mathcal{T}}\circ \xi = \mathrm{id}_Z$.
Thus \(Z\) is a retract of \([T, Z]_{\mathcal{T}}\).
It follows from Proposition~\ref{prop:retractive} that if \([T, Z]_\mathcal{T}\) is a strongly $d$-space (strongly well-filtered, strong $R$-space, respectively), then \(Z\) is also a strongly $d$-space (strongly well-filtered, strong $R$-space, respectively).
\end{proof}

However, the converse of Proposition~\ref{prop:Xsq-not-sd} does not hold, as shown below.

\begin{theorem}\label{prop:finite-discrete-isbell}
Let $T$ be a finite discrete space with $n$ points $\{t_{1}, t_{2}, \ldots, t_{n}\}$ and let $Z$ be any
$T_{0}$-space. Then $[T,Z]_{\mathcal{I}}$ is homeomorphic to $Z^{n}=Z\times\cdots\times Z$ $(product~ topology)$, via the
evaluation map
\[
h:[T,Z]_{\mathcal{I}}\longrightarrow Z^{n},\qquad f\mapsto (f(t_{1}), f(t_{2}), \ldots, f(t_{n})).
\]
\end{theorem}

\begin{proof}
As $T$ is discrete, every function $T\to Z$ is continuous; hence, the underlying set
of $[T,Z]_{\mathcal{I}}$ is $Z^{T}$, and we identify $Z^{n}\cong Z^{T}=\prod_{x\in T}Z$. In
particular, $h$ is a bijection.

Obviously, the lattice $\OO(T)=\mathcal P(T)$ is finite, then its Scott topology is the
Alexandroff topology and the Scott open subsets of $\mathcal P(T)$ are exactly
the upward closed families of subsets of $T$. Since $\mathcal P(T)$ is
finite, every upper set $\mathcal{U}$ is the union of the principal filters generated by
its finitely many minimal elements, i.e.,
$\mathcal{U}=\bigcup_{M\in\min \mathcal{U}}\mathord{\uparrow}M$ with
$\mathord{\uparrow}M=\{A\subseteq T:M\subseteq A\}$.
For $V\in\OO(Z)$,
$$N(\mathcal{U}\leftarrow V)=\{f\in Z^{T}: f^{-1}(V)\in \mathcal{U}\}
=\bigcup_{M\in\min \mathcal{U}}\{f\in Z^{T}:M\subseteq f^{-1}(V)\}
=\bigcup_{M\in\min \mathcal{U}}(\bigcap_{x\in M}N_x(V)),
$$
where $N_x(V):=\{f\in Z^{T}:f(x)\in V\}$. Thus every subbasic set of the
Isbell topology is a finite union of finite intersections of the sets
$N_x(V)$. Conversely, for the upper set
$\mathcal{U}=\mathord{\uparrow}\{x\}$,
\[
N(\mathord{\uparrow}\{x\}\leftarrow V)=\{f\in Z^{T}:f(x)\in V\}=N_x(V).
\]
Hence, the Isbell topology on $(T,Z)$ equals the topology generated by the
subbasis $\{N_x(V):x\in T,\ V\in\OO(Z)\}$, which is the topology of pointwise
convergence, i.e., the initial topology for the evaluation maps
$\mathrm{ev}_x(f)=f(x)$.

Let $\pi_x:Z^{T}\to Z$ denote the $x$-th projection. The product topology on
$Z^{T}$ is generated by the subbasis
$\{\pi_x^{-1}(V):x\in T,\ V\in\OO(Z)\}$, and
\[
h^{-1}\bigl(\pi_x^{-1}(V)\bigr)=\{f:f(x)\in V\}=N_x(V)
\]
for all $x\in T$ and $V\in\OO(Z)$. As a result, $h$ is continuous and open. We conclude that $h$ is a
homeomorphism and $[T,Z]_{\mathcal{I}}\cong Z^{n}$.
\end{proof}

Let $\two$ be equipped with the discrete topology $\tau_\two=\{\emptyset, \{0\}, \{1\}, \two \}$. For any $T_{0}$-space $Z$,
denote by $(\two, Z)$ the set of all continuous functions $\two\to Z$;
since $\two$ is discrete, every function from $\two$ to $Z$ is continuous. Let \([\two, Z]_{\mathcal{I}}\) denote the set \((\two, Z)\) endowed with the Isbell topology, where the sets
$N(H \leftarrow V)=\{f \in (\two, Z): f^{-1}(V) \in H\}$ with
$H\in \sigma(\tau_\two)$ and $V\in\mathcal{O}(Z)$ form a subbasis for Isbell topology.

The following corollary follows from Theorem \ref{prop:finite-discrete-isbell}.

\begin{corollary}\label{thm:func-homeo}
For every $T_{0}$-space $Z$, the  map
$h:[\two,Z]_{\mathcal{I}}\longrightarrow Z\times Z$ defined by $h(f)=(f(0),f(1))$
is a homeomorphism.
\end{corollary}

\begin{theorem}\label{thm:func-counter}
Let $X=A\sqcup B$. Then the function space $[\two,X]_{\mathcal{I}}$ is not a strong $d$-space, in particular, $[\two,X]_{\mathcal{I}}$ is neither a strong $R$-space nor strongly well-filtered.
\end{theorem}

\begin{proof}
 Let $Z=X$. Then $[\two,X]_{\mathcal{I}}\cong X\times X$ by Corollary~\ref{thm:func-homeo}, while $X\times X$ is not a strong $d$-space by Proposition~\ref{prop:Xp-not-sd} .
\end{proof}

\begin{corollary}
None of $[\two,X]_{\mathcal{P}}$, $[\two,X]_{\mathcal{K}}$, $[\two,X]_{\mathcal{C}}$ and $[\two,X]_{\mathcal{T}_*}$ is  a strong $d$-space.
\end{corollary}

\begin{proof}
Since $\two$ is an \(\alpha^*\)-space, it follows from Proposition ~\ref{proer} that $[\two,X]_{\mathcal{I}}=[\two,X]_{\mathcal{P}}=[\two,X]_{\mathcal{K}} = [\two,X]_{\mathcal{C}} = [\two,X]_{\mathcal{T}_*}$; accordingly none of them is a strong $d$-space.
\end{proof}

\begin{theorem}\label{prop:topologies-coincide}
The Scott topology and
the Isbell topology on  $(\two, X)$ coincide, both equal the product topology on $X\times X$.
\end{theorem}

\begin{proof}
We equip $(\two,X)$ with the pointwise order:
\[
f\leq g
\quad\Longleftrightarrow\quad
f(t)\leq_X g(t)\text{ for every }t\in\two.
\]
We denote by $[\two,X]_{\sigma}$ the set $(\two,X)$
equipped with the Scott topology of this order.	
	
Since $\two$ is discrete, $(\two,X)=X^{\two}=X\times X$ as a set. Write
$\sigma(X)$ for the topology of $X$, which is the Scott topology of the dcpo
$\mathbb{N}_\infty^{+}\sqcup Q$.  Since \(X\) actually carries the Scott topology and \(X\) is locally compact, we have \(\sigma(X \times X)=\sigma(X) \times \sigma(X)\) by Theorem~\ref{tmscott}.
The Isbell topology on $(\two,X)$ is the product topology
$X\times X$ by Theorem~\ref{prop:finite-discrete-isbell}.
Consequently, $[\two,X]_{\mathcal{I}}=[\two,X]_{\sigma}$, and both are equal to the product topology on $X\times X$.
\end{proof}

\begin{corollary}
$[\two,X]_{\sigma}$ is not a strong $d$-space.
\end{corollary}

\section{Conclusions}

With the aid of $X=A\sqcup B$, we arrive at the following conclusions:
\begin{itemize}
\item Neither the Smyth power space $\PS(X)$ nor Scott power space $P_\sigma(X)$ is a  strong $d$-space, thereby providing negative answers to Questions 7.3 and 7.7 of \cite{Xu2026}, Questions 7.1 and 7.2 of~\cite{WenBaoXu2026};
\item For some $T_{0}$-space $\two$, even if \(X\) is a strong \(R\)-space, the function space \([\two,X]_{\mathcal{T}}\) is not a strong \(d\)-space, where $\mathcal{T}=\mathcal{P}, \mathcal{K}, \mathcal{I}$, $\mathcal{C}$ or $\sigma$.
\end{itemize}

\section*{Acknowledgments}

 This work is supported by the Natural Science Foundation of Hunan Province (No. 2026JJ81209), National Natural Science Foundation of China--Mathematics Tianyuan Fund (No. 2111300034) and Science and Technology Project of Hebei Education Department (No. 25113159).


\begin{thebibliography}{99}

\bibitem{En}
R. Engelking, General Topology, Revised and completed edition, Sigma Series in Pure Mathematics 6, Heldermann Verlag, Berlin, 1989.

\bibitem{Ershov2020}
Y. Ershov, M. Schwidefsky, On function spaces, Siberian Electronic Mathematical Reports, 17: 999--1008, 2020.

\bibitem{Ershov2022}
Y.~Ershov, M.~Schwidefsky, On function spaces. II,
Siberian Electronic Mathematical Reports 19(2) (2022) 815--834.

\bibitem{Ershov2024}
Y. Ershov, M. Schwidefsky, On function spaces. III, Lobachevskii Journal of Mathematics, 45(4): 1493--1498, 2024.

\bibitem{Gierz}
G.~Gierz, K.~H. Hofmann, K.~Keimel, J.~D. Lawson, M.~W. Mislove, D.~S. Scott,
\emph{Continuous Lattices and Domains}, Encyclopedia of Mathematics and its
Applications, vol.~93, Cambridge University Press, Cambridge, 2003.

\bibitem{Goubault}
J.~Goubault-Larrecq, \emph{Non-Hausdorff Topology and Domain Theory}, New
Mathematical Monographs, vol.~22, Cambridge University Press, Cambridge, 2013.

\bibitem{HeckmannKeimel}
R.~Heckmann, K.~Keimel, Quasicontinuous domains and the Smyth powerdomain,
Electron.~Notes Theor.~Comput.~Sci.~298 (2013) 215--232.

\bibitem{Isbell}
J.~R. Isbell, Function spaces and adjoints, Math.~Scand.~36 (1975) 317--339.

\bibitem{LawsonXu2024}
J.~Lawson, X.~Xu, $T_0$-spaces and the lower topology, Math.~Struct.~Comput.
Sci.~34 (2024) 467--490.

\bibitem{Liu2021}
B. Liu, Q. Li, W.K. Ho, On function spaces related to d-spaces, Topology and its Applications, 300, 107757, 2021.

\bibitem{WenBaoXu2026}
X.~Wen, M.~Bao, X.~Xu, On strong $R$-spaces, arXiv preprint
arXiv:2607.20294 (2026).

\bibitem{XiLawson2017}
X.~Xi, J.~Lawson, On well-filtered spaces and ordered sets,
Topology and its Applications 228 (2017) 139--144.

\bibitem{Xu2016}
X.~Xu, \emph{Order and Topology}, Science Press, Beijing, 2016.

\bibitem{Xu2026}
X.~Xu, Strongly well-filtered spaces and strong $d$-spaces, Topology and its
Applications 377 (2026) 109654.

\bibitem{xubao2022}
X.~Xu, M.~Bao, X.~Zhang, On function spaces equipped with Isbell topology
and Scott topology, Mathematical Structures in Computer Science
32(8) (2022) 1099--1116.

\bibitem{XuWenXi2023}
X.~Xu, X.~Wen, X.~Xi, Scott topology on Smyth power posets,
Mathematical Structures in Computer Science 33(9) (2023) 832--867.

\bibitem{XuYangChen2026}
X.~Xu, Y.~Yang, L.~Chen, Strong well-filteredness of upper topology on
sup-complete posets, Math.~Struct.~Comput.~Sci.~36 (2026) e19, 1--11.

\bibitem{XuYang2021}
X.~Xu, Z.~Yang, Coincidence of the upper Vietoris topology and the Scott
topology, Topology and its Applications 288 (2021) 107480.

\bibitem{XuZhao2020}
X.~Xu, D.~Zhao, On topological Rudin's lemma, well-filtered spaces and sober
spaces, Topology and its Applications 272 (2020) 107080.

\bibitem{XuZhao2021}
X.~Xu, D.~Zhao, Some open problems on well-filtered spaces and sober spaces,
Topology and its Applications 301 (2021) 107540.
\end{thebibliography}
\end{document}